\documentclass[a4paper]{amsart}
\usepackage{amsmath,amsfonts,latexsym,amssymb,amsthm, verbatim}
\usepackage{url}

\newtheorem{tm}{Theorem}
\newtheorem{proposition}[tm]{Proposition}
\newtheorem{lemma}[tm]{Lemma}
\newtheorem{corollary}[tm]{Corollary}
\theoremstyle{definition}

\newtheorem{question}[tm]{Question}
\theoremstyle{remark}
\newtheorem*{acknowledgments}{Acknowledgments}
\newtheorem*{remark}{Remark}

\newcommand{\Q}{\mathbb{Q}}
\newcommand{\Qbar}{\overline{\Q}}
\newcommand{\C}{\mathbb{C}}
\newcommand{\Z}{\mathbb{Z}}

\newcommand{\F}{\mathbb{F}}

\newcommand{\ab}{\mathrm{ab}}
\newcommand{\tors}{\mathrm{tors}}

\DeclareMathOperator{\End}{End}
\DeclareMathOperator{\Gal}{Gal}
\DeclareMathOperator{\GL}{GL}
\DeclareMathOperator{\id}{id}
\DeclareMathOperator{\SL}{SL}
\DeclareMathOperator{\PSL}{PSL}
\DeclareMathOperator{\Hom}{Hom}
\DeclareMathOperator{\Res}{Res}

\DeclareMathOperator{\rk}{rk}
\DeclareMathOperator{\ord}{ord}

\def\lsup#1#2{{}^{#1}\!{#2}}

\begin{document}

\title{The growth of the rank of Abelian varieties upon extensions}
\author{Peter Bruin}
\address{Institut f\"ur Mathematik\\ Universit\"at Z\"urich\\ Winterthurerstrasse 190\\ 8057 Z\"urich\\ Switzerland}
\email{peter.bruin@math.uzh.ch}
\author{Filip Najman}
\address{Department of Mathematics\\ University of Zagreb\\ Bijeni\v cka cesta 30\\ 10000 Zagreb\\ Croatia}
\email{fnajman@math.hr}
\thanks{P.B. was supported by Swiss National Science Foundation grants 124737 and 137928, and by the Max-Planck-Institut f\"ur Mathematik, Bonn.\\
\indent F.N. was supported by the Ministry of Science, Education, and Sports, Republic of Croatia, grant 037-0372781-2821.}

\begin{abstract}
We study the growth of the rank of elliptic curves and, more generally, Abelian varieties upon extensions of number fields.

First, we show that if $L/K$ is a finite Galois extension of number fields such that $\Gal(L/K)$ does not have an index 2 subgroup and $A/K$ is an Abelian variety, then $\rk A(L)-\rk A(K)$ can never be 1. We obtain more precise results when $\Gal(L/K)$ is of odd order, alternating, $\SL_2(\F_p)$ or $\PSL_2(\F_p)$. This implies a restriction on $\rk E(K(E[p]))-\rk E(K(\zeta_p))$ when $E/K$ is an elliptic curve whose mod $p$ Galois representation is surjective. Similar results are obtained for the growth of the rank in certain non-Galois extensions.

Second, we show that for every $n\ge2$ there exists an elliptic curve $E$ over a number field $K$ such that $\Q\otimes\End_\Q\Res_{K/\Q} E$ contains a number field of degree~$2^n$. We ask whether every elliptic curve $E/K$ has infinite rank over $K\Q(2)$, where $\Q(2)$ is the compositum of all quadratic extensions of~$\Q$. We show that if the answer is yes, then for any $n\ge2$, there exists an elliptic curve $E/K$ admitting infinitely many quadratic twists whose rank is a positive multiple of $2^n$.
\end{abstract}

\maketitle
\begin{comment}
The second result concerns $K$-curves and says that, given an integer $m$, we can construct an elliptic $K$-curve $E$ over a number field $F$ such that $\rk E(F)$ is divisible by $2^m$. Furthermore, the same is true for $\rk E(L)$ if $L$ is a finite Galois extension of $K$ containing $L$ as a subfield or if $L$ can be written as $F\otimes_K N$ for some number field $N$. Under some assumptions, this implies the existence of elliptic curves over number fields with infinitely many twists of large rank.
\end{comment}

\section{Introduction}

Let $A$ be an Abelian variety over a number field~$K$. By the Mordell--Weil theorem, the Abelian group $A(K)$ of $K$-rational points of $A$ is finitely generated, so it is of the form $T\oplus \Z^r$, where $T$ is the torsion group and $r$ is the rank of $A$.
%In the case when $A$ is an elliptic curve, much is known about the possible structures of the torsion group over various number fields, but the rank remains a much more mysterious quantity.
%In particular, it is an important open question whether the rank of an elliptic curve over $\Q$ can be arbitrarily large.
%
In this paper, we study the growth of the rank of Abelian varieties upon extensions of number fields.

In Section \ref{sec2}, we consider an Abelian variety $A$ over a number field $K$ and a finite Galois extension $L/K$. We show that the structure of the group $G=\Gal(L/K)$ can impose restrictions on the growth of the rank of~$A$ under base extension from $K$ to~$L$.
Suppose throughout this paragraph that $\rk A(L)$ is strictly greater than~$\rk A(K)$. We prove that if $G$ has odd order, then $\rk A(L)-\rk A(K)\geq p-1$, where $p$ is the smallest prime factor of $\#G$.
If $G$ is the alternating group $A_n$ with $n\ge5$, we obtain $\rk A(L)- \rk A(K)\geq n-1$. If $G$ is $A_3$ or $A_4$, then $\rk A(L)- \rk A(K)\geq 2$.
If $G$ is $\SL_2(\F_p)$ or $\PSL_2(\F_p)$ for a prime $p>2$, then $\rk A(L)- \rk A(K)\geq \frac{p-1}{2}$.
As a corollary, if $E/K$ is an elliptic curve and $p>2$ is a prime such that the mod $p$ Galois representation of $E$ is surjective, then $\rk E(K(E[p]))-\rk E(K(\zeta_p))$ is either zero or at least $\frac{p-1}{2}$.
We prove that $\rk A(L)- \rk A(K)\geq 2$ whenever $G$ does not have an index 2 subgroup.

Furthermore, for certain non-Galois extensions $L/K$ with Galois closure $M$, we exhibit non-trivial relations between $\rk A(M)- \rk A(K)$ and $\rk A(L)- \rk A(K)$.

In Section \ref{sec3}, we build on the ideas of \cite{bbd}, where it was proved that $\Q$-curves of certain type have additional structure, which forces them to have even rank over their field of definition. The additional structure of these curves can be seen in the endomorphisms of their restrictions of scalars. In a similar way we construct, for arbitrarily large~$n$, an elliptic curve $E_n$ defined over a number field $K_n$ such that the endomorphism algebra of the restriction of scalars $\Res_{K_n/\Q}E_n$ contains a cyclotomic field of degree~$2^n$.

Let $\Q(2)$ be the compositum of all quadratic extensions of~$\Q$. We ask the following question: does every elliptic curve $E$ over a number field~$K$ have infinite rank over $K\Q(2)$? We show that if the answer is yes, then an elliptic curve $E_n/K_n$ as above has infinitely many quadratic twists whose rank is a positive multiple of~$2^n$.

\begin{comment}
If we assume a certain conjecture about elliptic curves over $\Q(2)$ the compositum of all quadratic extensions of $\Q$, we prove that there exists an elliptic curve $E$ over a number field $F$ with infinitely many quadratic twists of rank at least 16. If we assume an even stronger conjecture, we obtain that for any integer $m$, there exists an elliptic curve $E$ over a number field $F$ with infinitely many quadratic twists of rank at least $2^m$.
\end{comment}

\section{Growth of the rank in extensions}
\label{sec2}
In this section we study how the rank of an Abelian variety can grow upon finite extensions.

\subsection{Galois extensions}

Let $A$ be an Abelian variety over $K$, and let $L$ be a finite Galois extension of $K$. We compare the ranks of $A(K)$ and $A(L)$ by looking at the action of $G=\Gal(L/K)$ on $\Q\otimes A(L)$.
In particular, we are interested in the possible $\Q$-dimensions of $\Q\otimes (A(L)/A(K))$. This is controlled by the dimensions of the non-trivial irreducible $\Q$-linear representations of~$G$. For example, the smallest non-trivial increase of the rank of $A$ when going from $K$ to~$L$ equals the smallest dimension of an irreducible non-trivial $\Q$-linear representation of $\Gal(L/K)$. We will study such constraints on the rank of $A(L)$ for various groups~$G$.

\begin{tm}
\label{cyclic}
Let $p>2$ be a prime, let $A$ be an Abelian variety defined over a number field $K$, and let $L$ be a Galois extension of $K$ such that $\Gal(L/K)\simeq \Z /p\Z$. Then $$\rk A(L)\equiv \rk A(K) \pmod{p-1}.$$
\end{tm}
\begin{proof}
First note that the fixed subspace of $\Q \otimes A(L)$ under $\Gal(L/K)$ corresponds to $\Q\otimes A(K)$, the dimension of which is $\rk A(K)$.

Over $\C$, the irreducible representations of the group $\Z/p \Z$ are the trivial representation and the $p-1$ representations corresponding to the non-trivial characters of $\Z/p \Z$. However, the smallest field over which the non-trivial characters are defined is $\Q(\zeta_p)$. In fact, there are two irreducible $\Q$-linear representations: the trivial representation and a representation of dimension $p-1$. This implies that $\rk A(L)-\rk A(K)=\dim_\Q(\Q\otimes(A(L)/A(K)))$ is a multiple of $p-1$.
\end{proof}

\begin{corollary}
Let $A$ be an Abelian variety defined over a number field $K$, and let $L/K$ be a Galois extension of odd degree $n$. Then
$$\rk A(L)= \rk A(K) \text{ or }\rk A(L)\geq \rk A(K)+p-1,$$
where $p$ is the smallest prime dividing $n$.
\end{corollary}
\begin{proof}
By the Feit--Thompson theorem \cite{ft}, $\Gal(L/K)$ is solvable, so there is a sequence of intermediate fields $K=K_0\subset K_1\subset\cdots\subset K_r=L$ such that each extension $K_i/K_{i-1}$ is cyclic of some prime degree $p_i\mid n$. Theorem \ref{cyclic} implies that $\rk A(L)-\rk A(K)$ is a linear combination of the $p_i-1$, and the claim follows.
\end{proof}

We now turn our attention to the case where $G$ is an alternating group.

\begin{tm}
\label{alternating}
Let $A$ be an Abelian variety defined over a number field $K$. Let $n\ge3$ and let $L/K$ be a Galois extension with group $A_n$. Then
$$
\rk A(L)=\rk A(K) \text{ or } \rk A(L)\geq \rk A(K) +
\begin{cases}
2& \hbox{if }n=4,\\
n-1& \hbox{if }n=3\hbox{ or }n\ge5.
\end{cases}
$$
\end{tm}
\begin{proof}
As $A_3 \simeq \Z/3\Z$, the case $n=3$ is already proved in Theorem \ref{cyclic}.

The group $A_4$ has two non-trivial complex representations of dimension 1, but their images involve third roots of unity and are therefore are not defined over $\Q$. Hence all non-zero, non-trivial representations of $A_4$ over $\Q$ have dimension $\geq 2$.

The group $A_5$ has two irreducible complex representations of dimension~3,
%(corresponding to its action as icosahedral symmetry)
but these involve fifth roots of unity. The minimal dimension of a non-trivial $\Q$-linear irreducible representation equals 4.

It is well known \cite[Exercise 5.5]{fh} that the dimension of the smallest irreducible $\C$-linear (and hence also $\Q$-linear) representation of $A_n$ is of dimension $n-1$ for $n>5$, completing the proof.
\end{proof}

\begin{remark}
In the setting of Theorem \ref{alternating} for $n=4$, as $A_4$ has one 2-dimensional and one 3-dimensional $\Q$-linear irreducible representation. As every positive integer apart from 1 can be written as the sum of multiples of 2 and 3, it follows that $\rk A(L)- \rk A(K)$ can a priori be any non-negative integer apart from 1.
\end{remark}

\begin{comment}
In a similar manner, one can show that in the setting of Theorem \ref{alternating}(2) for $n=5$, $A_5$ has irreducible $\Q$-linear representations of dimension 3, 4, and 5. As any number apart from 1 and 2 can be written as the sum of multiples of 3, 4, and 5, it follows that $\rk E(L)- \rk E(K)$ can a priori be any non-negative integer apart from 1 and 2.
\end{comment}

Here is a similar result about extensions with Galois group $\SL_2(\F_p)$ or $\PSL_2(\F_p)$.

\begin{tm}
\label{slpsl}
Let $A$ be an Abelian variety over a number field $K$, and let $L$ be a finite Galois extension with group $\SL_2(\F_p)$ or $\PSL_2(\F_p)$ for some prime $p>2$.
Then
$$\rk A(L)= \rk A(K) \text{ or } \rk A(L)\geq \rk A(K))+\frac{p-1}{2}.$$
\end{tm}
\begin{proof}
The minimal dimension of an irreducible non-trivial $\Q$-linear representation of $\SL_2(\F_p)$ is $(p-1)/2$ \cite[Chapter 5.2, pages 71--73]{fh}.
Similarly, the smallest dimension of a non-trivial irreducible $\Q$-linear representation of $\PSL_2(\F_p)$ is $(p-1)/2$ \cite[Exercise 5.10, page 71]{fh}.
\end{proof}

\begin{comment}
Using Theorem \ref{slpsl}, we show that for every elliptic curve $E$ over $K$ without complex multiplication and for almost all primes $p$, there is a constraint on the possible values of $\rk E(K(E[p]))-\rk E(K(\zeta_p))$, where $K(E[p])$ is the $p$-division field of $E$, i.e., the splitting field of $E[p]$ over~$K$.
\end{comment}

\begin{corollary}
\label{cyc}
Let $E$ be an elliptic curve over a number field $K$, and let $p>2$ be a prime such that the Galois representation
$$
\rho_p\colon \Gal(\overline K /K)\to\GL_2(\F_p).
$$
coming from the action of $\Gal(\overline K/K)$ on $E[p]$ is surjective. Let $L$ be a subfield of $K(E[p])$ with $[K(E[p]):L]\le 2$. Then
$$\rk E(L)= \rk E(K(\zeta_p)) \text{ or } \rk E(L)\geq \rk E(K(\zeta_p))+\frac{p-1}{2}$$
\end{corollary}
\begin{proof}
We note that $\rho_p$ factors through $\Gal(K(E[p])/K)$. By the properties of the Weil pairing, $K(E[p])$ contains $K(\zeta_p)$, and $\rho_p$ identifies $\Gal(K(E[p])/K(\zeta_p))$ with~$\SL_2(\F_p)$.
If $[K(E[p]):L]=2$, then $L$ is the fixed field of~$-I$, the unique element of order~2 in $\SL_2(\F_p)$, and $\rho_p$ identifies $\Gal(L/K(\zeta_p))$ with~$\PSL_2(\F_p)$.
Hence $\Gal(K(E[p])/K(\zeta_p))$ is isomorphic to $\SL_2(\F_p)$ if $L=K(E[p])$, and to $\PSL_2(\F_p)$ if $[K(E[p]):L]=2$. The claim now follows from Theorem \ref{slpsl}.
\end{proof}

\begin{remark}
By Serre's open image theorem \cite{ser}, if $E$ does not have complex multiplication, then the map $\rho_p$ is surjective for all but finitely many primes $p$.
\end{remark}

In view of the results proved in this section, it is natural to wonder what characterizes the groups~$G$ such that for Galois extensions $L/K$ with group~$G$, one can give a non-trivial lower bound on $\rk A(L)-\rk A(K)$, whenever $\rk A(L)-\rk A(K)>0$. The following theorem answers this question.

\begin{tm}
\label{general}
Let $L/K$ be a finite Galois extension of number fields such that $G=\Gal(L/K)$
does not contain a subgroup of index~$2$. Then for any Abelian variety $A$ over~$K$, either $\rk A(L)=\rk A(K)$ or $\rk A(L)\ge\rk A(K)+2$.
\end{tm}
\begin{proof}
As $G$ has no subgroup of index~2, there is no non-trivial homomorphism $G\rightarrow \Q^\times_\tors=\{ 1,-1 \}$. Therefore $G$ has no non-trivial irreducible representation of dimension~1 over~$\Q$.
\end{proof}

Note that none of the groups~$G$ considered above has an index 2 subgroup. If $L/K$ is a finite Galois extension for which $\Gal(L/K)$ does have an index 2 subgroup, then $L$ contains $K(\sqrt d)$ for some $d\in K^\times\setminus (K^\times)^2$. If $A$ is an Abelian variety over~$K$, we cannot exclude that $\rk A(L)=\rk A(K(\sqrt d))>\rk A(K)$. Now $\rk A(K(\sqrt d)) =\rk A(K) + \rk A^d(K)$, where $A^d$ is the quadratic twist of $A$ by $d$; in general, we cannot prove any restrictions on $\rk A^d(K)$. So index~2 subgroups form an obstruction for results of the type of Theorems \ref{cyclic}, \ref{alternating} and~\ref{slpsl}; in this sense, Theorem \ref{general} is best possible.

\begin{remark}
The above results rely on the decomposition of $A(L)$ into irreducible representations of $G=\Gal(L/K)$. More conceptually, they can be interpreted via a decomposition of the Weil restriction $B=\Res_{L/K}A_L$ in the category of Abelian varieties over~$K$ up to isogeny, namely
$$
B\sim\bigoplus_\rho B_\rho,
$$
where $\rho$ ranges over the irreducible $\Q$-linear representations of~$G$ and the group algebra $\Q[G]$ acts on $B_\rho$ through a simple quotient algebra $R_\rho$; see \cite[\S\thinspace3.4]{dn} and \cite[Theorem 4.5]{mrs}. Our results can then be explained by the fact that in the situations we consider, $R_\rho$ is strictly larger than~$\Q$ for all non-trivial $\rho$.
\end{remark}

%\textbf{Remark 1} One can easily prove the analogues of the results of this section for elliptic curves with complex multiplication. The only difference would be that, for an elliptic curve $E$ with $\Q \otimes \End E=R$, where $R$ is a quadratic imaginary field, one would need to consider the $R$-linear irreducible representations of $\Gal(L/K)$.

\subsection{Non-Galois extensions}

We start by recalling a bit of representation theory. Let $G$ be a finite group, and let $H\subseteq G$ a subgroup. For finite-dimensional $\Q$-linear representations $V$ of~$G$, we are interested in non-trivial relations between the dimensions of the $\Q$-vector spaces $V^G\subseteq V^H\subseteq V$.

Let $C_G$ denote the set of conjugacy classes of~$G$, and let $\chi_V$ denote the character of the representation~$V$, viewed as a function on~$C_G$. It is well known, as a special case of Schur's orthogonality relations, that the dimension of~$V^G$ equals
\begin{align*}
d_G(\chi_V)&=\frac{1}{\#G}\sum_{g\in G}\chi_V(\text{conjugacy class of }g)\\
&=\frac{1}{\#G}\sum_{c\in C_G}\#c\cdot\chi_V(c).
\end{align*}
We can write $\chi_V=\sum_{\chi\in X(G)}n_\chi\chi$, where $X(G)$ is the set of characters of irreducible $\Q$-linear representations of~$G$ and the $n_\chi$ are non-negative integers. Then we have
\begin{align*}
\dim V^G&=n_{\bf 1},\\
\dim V^H&=\sum_{\chi\in X(G)} n_\chi d_H(\chi),\\
\dim V&=\sum_{\chi\in X(G)} n_\chi d_{\{\id\}}(\chi).
\end{align*}

In the above notation, $d_{\{\id\}}(\chi)=\chi(\{\id\})$ is the dimension of the irreducible representation of~$G$ with character~$\chi$. The results obtained above for Galois extensions $L/K$ with group~$G$ are explained by the fact that in all the cases we considered, $d_{\{\id\}}(\chi)>1$ for all non-trivial irreducible representations $\chi$ of $G$ over~$\Q$.

The explanation of the following theorem is that the pairs $(G,H)$ we consider have the property that $d_{\{\id\}}(\chi)$ is strictly greater than $d_H(\chi)$ for all non-trivial $\chi$. Namely, we note that
\begin{align*}
\dim V^H-\dim V^G&=\sum_{\chi\ne{\bf 1}}n_\chi d_H(\chi),\\
\dim V-\dim V^H&=\sum_{\chi\in X(G)}n_\chi
(d_{\{\id\}}(\chi)-d_H(\chi)).
\end{align*}
If $\dim V^H>\dim V^G$, then $n_\chi$ is non-zero for some $\chi\ne{\bf 1}$, and the contribution of this $\chi$ in the formula for $\dim V-\dim V^H$ shows that $\dim V>\dim V^H$.

We apply the above observations to the Galois group of a normal closure of a non-Galois extension $L/K$ of number fields. For simplicity, we assume $[L:K]\le 4$.

\begin{tm}
Let $L/K$ be an extension of number fields, let $n=[L:K]$, let $M/K$ be a normal closure of $L/K$, and let $G=\Gal(M/K)$. Let $A$ be an Abelian variety over~$K$.
\begin{enumerate}
\item If $n=3$ and $G\simeq S_3$, then $\rk A(M)-\rk A(K)\ge2(\rk A(L)-\rk A(K))$.
\item If $n=4$ and $G\simeq A_4$, then $\rk A(M)-\rk A(K)\ge3(\rk A(L)-\rk A(K))$, and $\rk A(L)$ and $\rk A(M)$ have the same parity.
\item If $n=4$ and $G\simeq S_4$, then $\rk A(M)-\rk A(K)\ge3(\rk A(L)-\rk A(K))$.
\end{enumerate}
\end{tm}

\begin{proof}
Let $H=\Gal(M/L)\subseteq G$, so that $[G:H]=n$. We identify $G$ with a transitive subgroup of~$S_n$ acting on~$\{1,2,\ldots,n\}$, in such a way that $H$ is the stabilizer of~$1$. We put $V=A(M)$, so that $A(L)=V^H$ and $A(K)=V^G$.

First let $L/K$ be a non-cyclic extension of degree~3, so that $G=S_3$ and $H=\{\id,(23)\}\subset G$. The group~$S_3$ has three irreducible representations over~$\Q$ (the situation is the same as over~$\C$): the trivial representation~${\bf 1}$, the sign representation~$\epsilon$, and a unique two-dimensional representation~$\rho$, namely the obvious permutation representation of~$S_3$ on $\{(x_1,x_2,x_3)\in\Q^3\mid x_1+x_2+x_3=0\}$. One can check easily that the $H$-invariant subspaces of ${\bf 1}$, $\epsilon$, $\rho$ are of dimension 1, 0, 1, respectively. This implies that if
$$
V\simeq n_{\bf 1}\cdot{\bf 1}\oplus n_\epsilon\cdot\epsilon
\oplus n_\rho\cdot\rho,
$$
then
\begin{align*}
\dim V^G&=n_{\bf 1},\\
\dim V^H&=n_{\bf 1}+n_\rho,\\
\dim V&=n_{\bf 1}+n_\epsilon+2n_\rho.
\end{align*}
This is equivalent to (1).

Let $V_4$ denote the unique normal subgroup of order~$4$ of~$S_4$; more concretely,
$$
V_4=\{\id,(12)(34),(13)(24),(14)(23)\}\subset A_4\subset S_4.
$$

Let us now consider $G=A_4$. Then we have $H=\langle(234)\rangle$ and $G=V_4\rtimes H$. The group~$A_4$ has three irreducible representations over~$\Q$: the trivial representation~${\bf 1}$, the direct sum of the two non-trivial one-dimensional representations $\epsilon$ and~$\bar\epsilon$ of~$A_4/V_4$ (each of which is defined over~$\Q(\zeta_3)$), and the standard 3-dimensional representation $\rho_3$. One checks that the $H$-invariant subspaces of ${\bf 1}$, $\epsilon+\bar\epsilon$, $\rho_3$ are of dimension 1, 0, 1, respectively. This proves (2).

Finally, we consider $G=S_4$. Then we have $H\simeq S_3$ and $G=V_4\rtimes H$. The group~$S_4$ has five irreducible representations, both over~$\C$ and
over~$\Q$: the trivial representation~${\bf 1}$, the sign representation~$\epsilon$, a two-dimensional representation~$\rho_2$ arising via the surjection $S_4\to S_3$ from the two-dimensional representation~$\rho$ of~$S_3$, the standard 3-dimensional representation $\rho_3$, and the 3-dimensional representation $\epsilon\otimes\rho_3$. One checks that the $H$-invariant subspaces of ${\bf 1}$, $\epsilon$, $\rho_2$, $\rho_3$, $\epsilon\otimes\rho_3$ are of dimension 1, 0, 0, 1, 0, respectively. This proves (3).
\end{proof}

\begin{comment}
We now consider the following situation. Let
$$
G=D_4=\langle\rho,\sigma\mid\rho^4,\sigma^2,(\sigma\rho)^2\rangle
$$
and let $H$ be the subgroup $\langle\sigma\rangle$ of index~4 in~$G$.
\end{comment}

Finally, we note a curious property of certain quadratic twists of Abelian varieties over quadratic extensions of number fields.

\begin{proposition}
\label{prop-v4d4}
Let $L/K$ be a quadratic extension of number fields, and let $A$ be an Abelian variety of dimension~$g$ over~$L$. Let $B=\Res_{L/K} A$, and assume that $\Q\otimes\End_K B$ contains a number field of some degree~$n$. Let $\delta\in K^\times$, and let $M$ be the Galois closure of $L(\sqrt\delta)$ over~$K$.
\begin{enumerate}
\item The rank of $A(L)$ is divisible by~$n$.
\item If $\Gal(M/K)\simeq V_4$, then $\rk A^\delta(L)$ is divisible by~$n$.
\item If $\Gal(M/K)\simeq D_4$, then $2\rk A^\delta(L)$ is divisible by~$n$.
\end{enumerate}
\end{proposition}

\begin{proof}
Our assumption on $\End_K B$ implies that $\rk A(L)=\rk B(K)$ is divisible by~$n$, so (1) is clear.

If $\Gal(M/K)\simeq V_4$, the field~$M$ is equal to $L(\sqrt\delta)$, and $M$ is also of the form $L(\sqrt e)$ with $e\in K$. Hence
\begin{align*}
\rk A^\delta(L)&=\rk A(M)-\rk A(L)\\
&=\rk B(K(\sqrt e))-\rk B(K).
\end{align*}
By assumption, both terms on the right-hand side are divisible by~$n$, implying (2).

If $\Gal(M/K)\simeq D_4$, the field~$M$ is a quadratic extension of $L(\sqrt\delta)$. Let $M_0$ be the unique $V_4$-extension of $K$ contained in $M$. By looking at the irreducible representations of~$D_4$, one can show that there are non-negative integers $a$, $b$, $c$, $d$, $e$ such that
\begin{align*}
\rk A(M) &= a + b + c + d + 2e,\\
\rk A(M_0) &= a + b + c + d,\\
\rk A(L(\sqrt\delta)) &= a + c + e,\\
\rk A(L) &= a + c.
\end{align*}
We note that $L$, $M_0$, $M$ (but not $L(\sqrt\delta)$) are all of the form $L\otimes_K N$ for some number field~$N$. This implies that the ranks of $A(L)$, $A(M)$, and $A(M)$ are all divisible by~$n$. Therefore $2e$ is divisible by~$n$. Furthermore,
\begin{align*}
\rk A^\delta(L)&=\rk A(L(\delta))-\rk A(L)\\
&=e,
\end{align*}
which proves (3).
\end{proof}

\begin{comment}
In other words, Proposition \ref{prop-v4d4}(2) says that if $n$ is even, then $\rk A^\delta(K)$ is divisible by $n/2$, and if $n$ is odd, then $\rk A^\delta(K)$ is divisible by~$n$.
\end{comment}

\section{$\Q$-curves and ranks of twists}
\label{sec3}

The question whether the rank of an elliptic curve over $\Q$ can be arbitrarily large is one of the most important open problems concerning elliptic curves. Somewhat similar questions are: how large can the rank of a twist of a fixed elliptic curve $E/\Q$ be, and what is the largest $n$ such that $E$ has infinitely many twists with rank at least $n$? The best known result about the latter question for an arbitrary $E/\Q$ is that there exist infinitely many twists of $E$ with rank at least $2$ \cite{mes1}. There exist elliptic curves with infinitely many twists of rank at least $4$ \cite{mes2,rs}. If one assumes the parity conjecture, then there are also elliptic curves over $\Q$ with infinitely many quadratic twists of rank 5 \cite{rs}.

\begin{comment}
Note that the question how large the rank of a quadratic twist can be is equivalent to the question how large the growth of the rank can be upon quadratic extensions.

Although one can expect to obtain better results over number fields, there are no known results better than the ones stated above for $\Q$, as there is no known way to raise the rank of infinitely many twists upon extensions simultaneously.
\end{comment}

In this section, for arbitrarily large $n$, we construct elliptic curves $E_n$ over number fields~$K_n$ such that the endomorphism ring of the Weil restriction of scalars $\Res_{K_n/\Q}E_n$ contains an order in a number field of degree $2^n$. We also study the problem of constructing, for arbitrarily large $n$, elliptic curves over number fields admitting infinitely many quadratic twists whose rank is a positive multiple of~$2^n$. We ask the question whether every elliptic curve $E/K$ has infinite rank over $K\Q(2)$, where $\Q(2)$ is the compositum of all quadratic extensions of~$\Q$. A positive answer would imply that the elliptic curves $E_n/K_n$ just mentioned have infinitely many quadratic twists whose rank is a positive multiple of~$2^n$.

The ideas are inspired by \cite{bbd}, where it was proved that every elliptic curve $E$ with a point of order 13 or 18 over a quadratic field $K$ has even rank. The reason for this is that the endomorphism ring of $\Res_{K/\Q} E$ contains $\Z[\sqrt d]$, where $d$ is not a square. This forces $(\Res_{K/\Q} E)(\Q)\simeq E(K)$ to be a $\Z[\sqrt d]$-module. Hence, $E(K)$ is of even rank.
The result mentioned in the previous paragraph shows that one can similarly construct elliptic curves over number fields whose rank is divisible by integers larger than~2.

A \emph{$\Q$-curve} is an elliptic curve $E$ over~$\overline\Q$ that is $\overline \Q$-isogenous to $\lsup \sigma E$ for all $\sigma \in \Gal(\overline \Q /\Q)$.
An interesting property of $\Q$-curves is the fact that the rich structure of these curves has consequences for their rank. For example, the proof that all elliptic curves over quadratic fields with a point of order 13 or 18 have even rank \cite{bbd} uses the fact that all such curves are in fact $\Q$-curves.
A good and thorough account of the properties of the endomorphism algebras of the restrictions of scalars of $\Q$-curves can be found in \cite{que}.

\begin{comment}
An important property of $\Q$-curves is that they are modular. Namely, an elliptic curve $E/\Qbar$ is modular (in the sense that there exists a non-constant morphism $X_1(N)\rightarrow E$ for some positive integer $N$) if and only if $E$ is a $\Q$-curve \cite{els,kha,kw,rib1}.

It was proved by Elkies \cite{elk} that every $\Q$-curve without complex multiplication is $\overline \Q$-isogenous to a $\Q$-curve attached to a $\Q$-rational point on the curve $$X^*(N)=X_0(N)/B(N),$$ where $B(N)$ is the group of Atkin-Lehner involutions of $X_0(N)$.
\end{comment}

\begin{proposition}
\label{endcyc}
For every integer $n\geq 2$, there exists an elliptic curve $E$ over a number field $K$ such that $\Q\otimes\End_\Q (\Res_{K/\Q} E)$ contains a number field of degree~$2^n$.
\end{proposition}
\begin{proof}
Most of the work needed for this proposition has already been done in \cite[pages 309--312]{que}. Let
\begin{equation}
\label{ela}
E_a\colon y^2=x^3-3\sqrt a (4+5\sqrt a)x +2\sqrt a (2+14\sqrt a +11 a),
\end{equation}
where $a$ is a non-square rational number, be a member of the family of $\Q$-curves parametrized by $X^*(3)$ (the quotient of $X_0(3)$ by the Atkin--Lehner involution $w_3$); see \cite[page~309]{que}. Let $p$ be a prime such that $p\equiv 2 \pmod 3$ and $p\equiv 1 \pmod {2^n}$; there exists infinitely many such primes by the Chinese remainder theorem and Dirichlet's theorem on primes in arithmetic progressions. Note also that this prime satisfies $p\equiv 5 \pmod{12}$.

Let us write $\nu=\ord_2(p-1)$. Let $\epsilon$ be a Dirichlet character of order $2^\nu$ and conductor $4p$. Such a character exists, since $(\Z / 4p\Z)^\times \simeq \Z/2 \Z \oplus \Z/ (p-1) \Z$ and an element $(1,t)$, where $t$ is of order $2^\nu$, written in additive notation, has the desired properties. We write $K=K_\epsilon(\sqrt{-p})$, where $K_\epsilon$ is the splitting field of $\epsilon$.

As explained in \cite[page~312 (d)]{que}, under these assumptions, there exists an element $\gamma\in K^\times$ such that $E_{-p}^{\gamma}$ satisfies
$$\Q\otimes \End_\Q \Res_{K/\Q} E_{-p}^{\gamma}\simeq \Q (\zeta_{2^{\nu+1}},\sqrt 3).$$
Note that $\nu\geq n$, so $[\Q (\zeta_{2^{\nu+1}},\sqrt 3):\Q]\geq 2 \phi(2^{n+1})=2^{n+1}$. The claim follows.
\end{proof}

\begin{remark}
Let $\Q(2)$ be the compositum of all quadratic extensions of $\Q$. Note that the elliptic curve $E_a$ from \eqref{ela} is defined over a quadratic field. By \cite[Theorem 5]{iml}, $E_a(\Q(2))$ has infinite rank. (The statement of loc.\ cit.\ is that $E_a(\Q^\ab)$ has infinite rank, but the proof shows in fact that already $E_a(\Q(2))$ has infinite rank.) This implies that there exist infinitely many quadratic twists $E_a^d$ with $d$ a rational integer, pairwise non-isomorphic over $\Q(\sqrt a)$, such that $E_a^d(\Q(\sqrt a))$ has positive rank. Let $S$ be the set of such integers $d$. Since for any finite extension $F/\Q(\sqrt{a})$ the set of $d\in \Q$ with $\sqrt d \in F$ is finite, it follows that also over $F$ there are infinitely many $d\in F^\times/(F^\times)^2$ such that $E_a^d(F)$ has positive rank.
\end{remark}

If $E$ is a $\Q$-curve and $K\subset\overline\Q$ is a number field, we say that $E$ is \emph{completely defined over~$K$} if $E$ and all isogenies $E\rightarrow \lsup\sigma E$, for all $\sigma \in \Gal(\overline \Q / \Q)$, are defined over $K$.

\begin{comment}
\begin{proposition}
\label{mainprop}
Let $E$ be a $\Q$-curve completely defined over $K$, such that $\Q\otimes\End_\Q (\Res_{K/\Q} E)$ contains a number field $B$. Then for every Galois extension $M$ of $\Q$ containing $K$, $\Q\otimes\End_\Q (\Res_{M/\Q} E)$ contains $B$.
\end{proposition}

\begin{proof}
Observe that
$$\End_\Q(\Res_{M/\Q}E) \simeq \prod_{\sigma\in\Gal(M/\Q)}\Hom_M(\lsup\sigma E ,E).$$
Since if $\sigma,\tau\in \Gal(M/\Q)$ restrict to the same element of $\Gal(K/ \Q)$, then $\lsup\sigma E =\lsup\tau E$, it follows that
$$\prod_{\sigma\in\Gal(M/\Q)}\Hom_M(\lsup\sigma E ,E)\simeq
\prod_{\tau\in\Gal(M/K)}\prod_{\sigma\in\Gal(K/\Q)}\Hom_M(\lsup\sigma E ,E).$$
The assumption that $E$ is completely defined already over $K$ implies
$$\Hom_M(\lsup \sigma E, E)=\Hom_K(\lsup \sigma E, E)$$
for every $\sigma\in \Gal(M/\Q)$.
It follows that
$$\Q\otimes\End_\Q(\Res_{M/\Q}E)$$
is equal to $[M:K]$ copies of the number field $B$, completing the proof.
\end{proof}
\end{comment}

\begin{proposition}
\label{ext}
Let $E$ be a $\Q$-curve completely defined over a number field~$K$ such that $\Q\otimes \End_\Q(\Res_{K/\Q} E)$ contains a number field~$B$. For every number field~$N$ which can be written as $K\otimes_\Q N'$ for some number field $N'$, $\Q\otimes\End_\Q (\Res_{N/\Q} E)$ is a $B$-vector space.
\end{proposition}

\begin{proof}
Let $N'$ be a number field such that $N=K\otimes_\Q N'$ is also a field. Then
$$E(N)=E(K\otimes_\Q N')\simeq\Res_{K/\Q}E(N').$$
As $\Q\otimes \End_{N'} (\Res_{K/\Q}E)$ contains $B$, it follows that $\Q\otimes E(N) \simeq \Q\otimes \Res_{K/\Q}E(N')$ has a natural $B$-vector space structure.
\end{proof}

\begin{comment}

\begin{remark}
There exist number fields satisfying the conditions of Proposition \ref{ext}(2) but not Proposition \ref{ext}(3), and conversely. Let $K$ be a quadratic field which can be embedded in a $\Z/4\Z$-extension $M$ of $\Q$. Then $M$ cannot be written as $K\otimes_\Q N'$, where $N'$ is a number field. On the other hand, if $N'=\Q(\sqrt[3]{m })$ for some rational integer $m$, then $N=K\otimes_\Q N'$ is a field, but in general $N$ is not a Galois extension of $\Q$.
\end{remark}

Note also that the proof is effective in the sense that fields $F$ from Theorem \ref{mainprop} can be explicitly obtained, although it might be a bit unwieldy for large $n$, since it require computations over number fields of large degree and with elliptic curves with very large coefficients.

Finally, we will use the following trivial fact.
\begin{lemma}
\label{mainlem}
Let $E$ be an elliptic curve defined over a number field $K$ and let $K/\Q$ be a finite Galois extension of fields such that $B=\Q\otimes \End_\Q(\Res_{K/\Q} E)$ is a number field. The rank of $E(K)$ is divisible by $[B:\Q]$.
\end{lemma}

\end{comment}

Let $\Q(2)$ be the compositum of all quadratic extensions of $\Q$. The following question is a variant of \cite[Question 2]{iml}.

\begin{question}
\label{as2}
Does every elliptic curve over a number field~$K$ have infinite rank over $K\Q(2)$?
\end{question}

\begin{comment}
We expect the answer to be yes.
\end{comment}

\begin{remark}
Suppose that every Abelian variety over $\Q$ has infinite rank over $\Q(2)$; this is a variant of \cite[page 127, Problem]{fj}. Then by taking the Weil restriction, we obtain a positive answer to Question \ref{as2}.
\end{remark}

\begin{comment}
We are now ready to prove the main theorem of this section.
\end{comment}

\begin{tm}
\label{maintm}
Suppose that Question \ref{as2} has a positive answer. Let $n$ be a positive integer. There exist a number field $K$ and an elliptic curve $E$ over $K$ possessing infinitely many twists $E^d$ over $K$ such that $\rk E^d(K)$ is a positive multiple of $2^n$.
\end{tm}
\begin{proof}

As shown in Proposition \ref{endcyc}, there exists an elliptic curve $E$ over a number field $K$ such that $\Q\otimes\End_\Q (\Res_{K/\Q} E)$ contains a number field $B$ of degree $2^n$. It follows that the rank of $E(K)=(\Res_{K/\Q}E)(\Q)$ is divisible by $2^n$.

By assumption, $E$ has infinitely many (pairwise non-isomorphic) quadratic twists $E^d$, with $d$ a square-free integer, such that $\rk E^d(K)>0$. Let $E^d$ be such a twist, with $\sqrt d \not\in K$. As $K\otimes_\Q \Q(\sqrt d)\simeq K(\sqrt d)$ is a field (see for example \cite[Theorem 2.2]{ct}), it follows from Proposition \ref{ext} that the rank of $E(K(\sqrt d))=(\Res_{K/\Q}E)(\Q(\sqrt d))$ is divisible by $2^n$. But
$$\rk E(K(\sqrt d))=\rk E(K)+\rk E^d(K),$$
from which it follows that $\rk E^d(K)$ is divisible by $2^n$.
\end{proof}

\begin{acknowledgments}
We are greatly indebted to Jordi Quer for help with the proof of Proposition \ref{endcyc}. We are grateful to Andrej Dujella, Ivica Gusi\'c, and Matija Kazalicki for many helpful comments and motivating discussions.
\end{acknowledgments}

\end{document}